\documentclass[10pt]{article}
\usepackage[left=1in,right=1in]{geometry}
\usepackage[english]{babel}
\usepackage[T1]{fontenc}
\usepackage{enumerate}
\usepackage[dvipsnames]{xcolor}

\def\mF{\mathcal{F}}

\def\zP{\mathcal{P}}

\newenvironment{claimproof}[1][\proofname]{\proof[#1]}{\endproof}
\usepackage{color}
\usepackage{url}
\usepackage{amsmath,amssymb, amsthm}
\usepackage{hyperref}
\usepackage{authblk}
\usepackage[noabbrev,capitalise,nameinlink]{cleveref} 
\usepackage{comment}

\newtheorem{theorem}{Theorem}

\newtheorem{lemma}[theorem]{Lemma}
\newtheorem{corollary}[theorem]{Corollary}

\theoremstyle{remark}

\title{Small hitting sets for longest paths and cycles}
\author[1]{Sergey Norin\thanks{The research of S.N. was supported by an NSERC Discovery grant.}}
\affil[1]{McGill University, Montr\'eal, Canada.}
\author[2]{Raphael Steiner\thanks{The research of R.S. was supported by the Ambizione grant 216071 of the Swiss National Science Foundation.}}
\affil[2]{ETH Z\"{u}rich, Zurich, Switzerland.}
\author[3]{Stéphan Thomassé\thanks{The research of S.T. was supported by the French National Research Agency under research grant ANR GODASse ANR-24-CE48-4377.}}
\affil[3]{Univ Lyon, EnsL, UCBL, CNRS, LIP, F-69342, LYON Cedex 07, France.}
\author[4]{Paul Wollan}
\affil[4]{Sapienza Universit\`a di Roma, Rome, Italy.}

\begin{document}

\maketitle
\begin{abstract}
Motivated by an old question of Gallai (1966) on the intersection of longest paths in a graph and the famous conjectures of Lov\'{a}sz (1969) and Thomassen (1978) on the maximum length of paths and cycles in vertex-transitive graphs, we present improved bounds for the parameters $\mathrm{lpt}(G)$ and $\mathrm{lct}(G)$, defined as the minimum size of a set of vertices in a graph $G$ hitting all longest paths (cycles, respectively).
First, we show that every connected graph $G$ on $n$ vertices satisfies $\mathrm{lpt}(G)\le \sqrt{8n}$, and $\mathrm{lct}(G)\le \sqrt{8n}$ if $G$ is additionally $2$-connected. This improves a sequence of earlier bounds for these problems, with the previous state of the art being $O(n^{2/3})$.
Second, we show that every connected graph $G$ satisfies $\mathrm{lpt}(G)\le O(\ell^{5/9})$, where $\ell$ denotes the maximum length of a path in $G$. As an application of this latter bound, we present further progress towards Lov\'{a}sz' and Thomassen's conjectures: We show that every connected vertex-transitive graph of order $n$ contains a cycle (and path) of length $\Omega(n^{9/14})$. This improves the previous best bound of the form $\Omega(n^{13/21})$.
Interestingly, our proofs employ several concepts and results from structural graph theory, such as a result of Robertson and Seymour (1990) on transactions in societies and Tutte's $2$-separator theorem.
\end{abstract}

\section{Introduction}

A well-known folklore fact states that in a connected graph $G$, any two longest paths\footnote{Throughout this paper, when speaking of a \emph{longest path (cycle)}, we mean a path (or cycle) that attains the maximum possible length among all paths (or cycles) in the graph.} must share a common vertex. This motivated Gallai~\cite{Gal66} in 1966 to ask if an even stronger statement could be true, namely whether in every connected graph there is a vertex that belongs to all longest paths. To discuss this question in a more general context, it will be convenient to denote by $\mathrm{lpt}(G)$ and $\mathrm{lct}(G)$ respectively, the minimum size of a \emph{longest path transversal} and the minimum size of a \emph{longest cycle transversal} of $G$, respectively, that is, a subset of vertices $S\subseteq V(G)$ such that every longest path (cycle) in $G$ contains at least one vertex from $S$. Gallai's question then can be restated as asking whether $\mathrm{lpt}(G)=1$ for every connected graph~$G$. While the answer has been shown to be positive for many special classes of graphs, see~\cite{balister04,Cerioli20,Cerioli202,Chen17,Der13,golan16,gutierrez21,harvey23,joos15,klavvzar90, Long23,Wiener18}, it is known to be negative in general, as was first shown by Walther~\cite{Wal69} in 1969, who constructed a connected graph $G$ on $25$ vertices with $\mathrm{lpt}(G)=2$. The currently smallest example of a graph $G$ with the same property has $12$ vertices, as found by Walther~\cite{Wal78} and Zamfirescu~\cite{Zam76}. A graph with no longest path transversal of size 2 was found in 1974 by Grünbaum~\cite{Gru74}. The smallest known construction of a graph with $\mathrm{lpt}(G)=3$ is due to Zamfirescu~\cite{Zam76}. Quite amazingly, no connected graph $G$ satisfying $\mathrm{lpt}(G)>3$ has been found in the last half century since Gr\"{u}nbaum's example. In particular, the tantalizing question, raised independently by Walther and Zamfirescu~\cite{Zam72}, whether there exists an absolute constant $c>0$ such that every connected graph $G$ satisfies $\mathrm{lpt}(G)\le c$, remains widely open. The situation and state of the art for longest cycle transversals is analogous to that of longest path transversals, except that here the natural setting to study the question is given by $2$-connected graphs: In $2$-connected graphs any two longest cycles are guaranteed to intersect, while this is not in general the case in connected graphs. Further, it is easy to construct examples of connected $n$-vertex graphs where every longest cycle transversal must have  size $\Theta(n)$. We refer to the survey articles~\cite{shabbir13,Zam01} for more details on the known lower bound examples for longest cycle transversals.

Let $P,C:\mathbb{N}\rightarrow \mathbb{N}$ denote the smallest functions such that every $n$-vertex ($2$-)connected graph has a longest path (cycle) transversal of size at most $P(n)$ (at most $C(n)$), respectively. As mentioned above, the best known lower bounds on $P(n)$ and $C(n)$ are constant, which raises the question whether one can conversely prove good upper bounds on $P(n)$ and $C(n)$. Another motivation for upper-bounding $P(n)$ and $C(n)$ was recently noted by Groenland et al.~\cite{groen24} (see also~\cite{Babai79,Devos23}).  Two famous conjectures due to Lov\'{a}sz~\cite{Lovasz70} from 1969 and Thomassen (cf.~\cite{barat10}) from 1978 postulate that all (sufficiently large) connected vertex-transitive graphs admit a Hamiltonian path (cycle). These tantalizing conjectures have received a lot of attention by many researchers in the last 50 years, but yet remain widely open today. In particular, it is not known whether every connected vertex-transitive graph of order $n$ contains a path of length $\Omega(n)$. We refer to~\cite{Alspach81,Kutnar09} for surveys on the partial results towards Lov\'{a}sz' and Thomassen's conjectures. One can observe (cf.~\cite[Proposition~5.1]{groen24}) that every connected vertex-transitive graph of order $n$ contains a path of length at least $\frac{n}{P(n)}-1$ and a cycle of length at least $\frac{n}{C(n)}$. Thus, any progress on the asymptotic upper bounds on $P(n)$ and $C(n)$ can be converted into progress on finding long paths and cycles in vertex-transitive graphs. 

The first non-trivial upper bounds on $P(n)$ and $C(n)$ were obtained by Rautenbach and Sereni~\cite{Rau14}, who showed that $P(n)\leq \left\lceil\frac{n}{4}-\frac{n^{2/3}}{90}\right\rceil$ and $C(n)\le \left\lceil\frac{n}{3}-\frac{n^{2/3}}{36}\right\rceil$. Recently, both of these bounds were asymptotically improved, first to the sublinear bounds $P(n)=O(n^{3/4})$ and $C(n)=O(n^{3/4})$ by Long, Milans and Munaro~\cite{Lon21}, and then further to $P(n)=O(n^{2/3})$ and $C(n)=O(n^{2/3})$ by Kierstead and Ren~\cite{Kie23}. In this paper, we further reduce both upper bounds to $O(\sqrt{n})$. Additionally, we provide a new upper bound for longest path transversals in terms of the length $\ell$ of a longest path instead of the number of vertices. 
\begin{theorem}\label{thm:main}
\noindent\begin{enumerate}[(i)]
    \item Let $G$ be a connected graph on $n\ge 2$ vertices and let $\ell$ denote the maximum length of a path in $G$. Then $$\mathrm{lpt}(G)\le \min\{\sqrt{8n},33\cdot \ell^{5/9}\}.$$
    \item For every $2$-connected graph $G$ on $n$ vertices, we have $\mathrm{lct}(G)\le \sqrt{8n}$.
\end{enumerate}
\end{theorem}
As a consequence of Theorem~\ref{thm:main}~$(i)$ we deduce a new lower bound for the length of longest cycles in connected vertex-transitive graphs with $n$ vertices, thus making further progress on Lov\'{a}sz' and Thomassen's conjectures~\cite{Lovasz70,barat10}.
\begin{corollary}\label{cor:lovasz}
    Every connected vertex-transitive graph on $n\ge 3$ vertices contains a cycle (and path) of length at least $\Omega(n^{9/14})$.
\end{corollary}
Corollary~\ref{cor:lovasz} improves a sequence of previous lower bounds on the length of longest paths and cycles in vertex-transitive graphs starting with a paper of Babai~\cite{Babai79} from 1979, who gave a lower bound of $\Omega(\sqrt{n})$. This long-standing bound was only recently further improved, first by DeVos~\cite{Devos23} to $\Omega(n^{3/5})$, and then to $\Omega(n^{13/21})$ by Groenland et al.~\cite{groen24}, which remained the state of the art until now. Our proof of Theorem~\ref{thm:main} differs significantly from the previous upper bound proofs in~\cite{Kie23,Lon21,Rau14} and employs several new ideas. Curiously, important steps of our proof make use of results from seemingly unrelated areas of structural graph theory, such as a result on \emph{transactions}  in \emph{societies} from the graph minors series due to Robertson and Seymour and the so-called \emph{$2$-separator theorem} due to Tutte.
\paragraph*{Organization.} In Section~\ref{sec:aux}, we prove several useful auxiliary statements that will be needed in the proof of Theorem~\ref{thm:main}. One of the lemmas, namely Lemma~\ref{lem:matchings}, requires a bit of a longer proof which is thus separately prepared and presented in Section~\ref{sec:matchings}. Finally, in Section~\ref{sec:thm} we put all ingredients together and present the proofs of Theorem~\ref{thm:main} and Corollary~\ref{cor:lovasz}.
\paragraph*{Notation and Terminology.} For $k\in\mathbb{N}$ we denote by $[k]:=\{1,\ldots,k\}$ the set of the first $k$ integers. Given a graph $G$, we denote by $V(G)$ its vertex set and by $E(G)$ its edge set. If $v\in V(G)$ then we denote by $N_G(v)$ the neighborhood of $v$ and by $d_G(v)$ its degree. For a subset $X\subseteq V(G)$ we denote by $G[X]$ the induced subgraph of $G$ with vertex set $X$, and we also denote $G-X:=G[V(G)\setminus X]$. For two subsets of vertices $A,B\subseteq V(G)$, an \emph{$A$-$B$-path} is a path in $G$ with first vertex in $A$, last vertex in $B$ and no internal vertices in $A\cup B$. We say that a set of vertices $S\subseteq V(G)$ \emph{separates} a pair of subsets $A,B$ if there exist no $A\setminus S$-$B\setminus S$-paths in the graph $G-S$. A \emph{separator} of $G$ is a set $S\subseteq V(G)$ such that $G-S$ is disconnected. Given a path $P$ or a cycle $C$ in a graph $G$, we denote by $|P|, |C|$ their lengths (i.e., number of edges). For two vertices $u$ and $v$, we denote by $\mathrm{dist}_G(u,v)$ their distance in $G$. A cycle $C$ in a graph $G$ is called \emph{geodetic} if $\mathrm{dist}_C(x,y)=\mathrm{dist}_G(x,y)$ for any two vertices $x,y \in V(C)$. If $G$ is a graph equipped with a weight function $\omega:V(G)\rightarrow \mathbb{R}$ on its vertices, and $H\subseteq G$ is a subgraph of $G$, then we denote by $\omega(H):=\sum_{v\in V(H)}{\omega(v)}$ the \emph{total weight} of $H$.

\section{Auxiliary statements}\label{sec:aux}
In this section, we prepare the proof of Theorem~\ref{thm:main} by collecting several useful lemmas. We start with the following statement, which contains a few variants of the aforementioned folklore-fact that any two longest paths in a connected graph intersect, that will be useful later. We say that a path $P$ in a graph $G$ with endpoints $u,v$ is \emph{locally-longest} if it has maximum length among all paths with endpoints $u,v$. 
\begin{lemma}\label{lem:intersect}
Let $G$ be a graph, and let each of $L_1, L_2$ be either a longest cycle or a locally-longest path in $G$. If there exist two vertex-disjoint $V(L_1)$-$V(L_2)$-paths in $G$, then $V(L_1)\cap V(L_2)\neq \emptyset$.
\end{lemma}
\begin{proof}
Suppose towards a contradiction that $V(L_1)\cap V(L_2)=\emptyset$. Let $R_1, R_2$ denote two vertex-disjoint $V(L_1)$-$V(L_2)$-paths in $G$. Let $d_1\ge 1$ be the distance of the endpoints of $R_1, R_2$ on $L_1$ and $d_2\ge 1$ the distance between the endpoints of $R_1, R_2$ on $L_2$. Since $L_1$ is a locally-longest path or a longest cycle, we must have that $d_1$ is at least as large as the length of the path formed by the union of $R_1, R_2$ and any segment of $L_2$ connecting the endpoints of $R_1$ and $R_2$ on $L_2$. In particular, we must have $d_1\ge d_2+|R_1|+|R_2|>d_2$. However, the symmetric argument shows that $d_2>d_1$, a contradiction, as desired. 
\end{proof}
In our proof we further make use of the following fact, saying that any cycle that is disjoint from a longest path or cycle may be separated from the latter using relatively few vertices.
\begin{lemma}\label{lem:separate}
Let $G$ be a graph, let $L$ be a longest path or a longest cycle in $G$, and let $C$ be a cycle in $G$ disjoint from $L$. Then $V(C)$ and $V(L)$ can be separated by a set of at most $\sqrt{2|L|}$ vertices. 
\end{lemma}
\begin{proof}
Assume towards a contradiction that this is not the case. Then, by Menger's theorem, there is a family $\mF$ of vertex-disjoint $V(C)$-$V(L)$-paths such that $|\mF|>\sqrt{2|L|}$. Then $|C|\ge |\mathcal{F}|$ and there must be two vertices $x,y\in V(L)$ which have distance at most $\frac{|L|}{|\mF|-1}$ in $L$ and which are both endpoints of some paths $P_x,P_y$ in $\mF$. Let $u$, $v$ denote the distinct endpoints of $P_x$ and $P_y$ on $C$, respectively. Then there exists a segment $Q$ of $C$ connecting $u$ to $v$ such that $|Q|\ge \frac{|C|}{2}\ge \frac{|\mF|}{2}$. Now consider the path between $x$ and $y$ obtained by concatenating $P_x, Q$ and $P_y$. It has length at least $|Q|+2\ge \frac{|\mF|+4}{2}$. On the other hand, since $L$ is a longest path or cycle, the length of this path cannot exceed the distance between $x$ and $y$ on $L$. Hence we obtain
$$\frac{|\mF|+4}{2}\le \mathrm{dist}_L(x,y)\le \frac{|L|}{|\mF|-1},$$ which implies $|L|\ge \frac{|\mathcal{F}|^2+3|\mathcal{F}|-4}{2}>\frac{|\mathcal{F}|^2}{2}>|L|$, the desired contradiction.
\end{proof}
The next lemma guarantees a nicely structured hitting set for all longest paths (cycles) in any ($2$-) connected graph. 
\begin{lemma}\label{lemma:nicehitting}
Let $G$ be a graph.
\begin{enumerate}[(i)]
    \item If $G$ is connected, then $\mathrm{lpt}(G)=1$ or there exists a cycle in $G$ intersecting all longest paths.
    \item If $G$ is $2$-connected, then there exists a cycle in $G$ intersecting all longest cycles.
\end{enumerate}
\end{lemma}
\begin{proof}\noindent
\begin{enumerate}[(i)]
        \item Let us denote by $\mathcal{P}$ the family of all longest paths in $G$. Recall that a \emph{block} of $G$ is a maximal subgraph that is $2$-connected or isomorphic to $K_2$. Let us consider the \emph{vertex-block-incidence graph} $T$ of $G$, i.e., the bipartite graph whose bipartition classes are $V(G)$ and the set of blocks of $G$, and a vertex in $G$ is adjacent to all blocks of $G$ that contain it. As is well-known, $T$ is a tree. Further, every $P\in \mathcal{P}$ naturally maps to a subtree $T(P)\subseteq T$ of $T$, induced by $V(P)$ and the set of all blocks intersected by $V(P)$. Since any two longest paths in $G$ intersect, we also have that the family $(T(P))_{P\in \mathcal{P}}$ of subtrees of $T$ is pairwise intersecting. Hence, by the Helly-property of subtrees of a tree, there exists a vertex or a block which is contained in $T(P)$ for every $P\in \mathcal{P}$. This implies that, in either case, there exists a block $B$ of $G$ such that $V(P)\cap V(B)\neq \emptyset$ for every $P\in \mathcal{P}$. Let us first consider the case that there exists some $P_0\in \mathcal{P}$ with $|V(P_0)\cap V(B)|=1$. Let $\{x\}:=V(P_0)\cap V(B)$. We now claim that $\{x\}$ is a longest path transversal. Indeed, for every $P\in \mathcal{P}$ we have $V(P)\cap V(B)\neq \emptyset$ and $V(P)\cap V(P_0)\neq \emptyset$. However, since $B$ is a block, we have that $\{x\}$ separates $V(P_0)$ from $V(B)$ in $G$, and hence we must have $x\in V(P)$ as well, as desired. Hence, $\mathrm{lpt}(G)=1$ in this case, concluding the proof. Next, suppose that $|V(P)\cap V(B)|\ge 2$ for every $P\in \mathcal{P}$. If $B$ contains no cycles, it must be a $K_2$. But then every $P\in \mathcal{P}$ contains both vertices of $B$, and hence one of these two vertices forms a longest path transversal, again resulting in $\mathrm{lpt}(G)=1$ and concluding the proof. Finally, suppose that $B$ contains a cycle. Let $C$ be a longest cycle in the graph $B$. Note that for every $P\in \mathcal{P}$, the intersection $P\cap B$ is a subpath of $P$ and hence (since $P$ is a longest path in $G$) must be a locally-longest path in $B$. Further, since $B$ is a $2$-connected graph and $|V(P\cap B)|\ge 2$, there exist two vertex-disjoint $V(P\cap B)$-$V(C)$-paths in $B$. Lemma~\ref{lem:intersect} now implies that $V(P\cap B)\cap V(C)\neq \emptyset$. Since $P\in \mathcal{P}$ was chosen arbitrarily, this shows that $V(C)$ is a longest path transversal, again confirming the claim. This concludes the proof.
        \item As $G$ is $2$-connected it contains a cycle. Let $C$ be a longest cycle in $G$. Lemma~\ref{lem:intersect} and Menger's theorem imply that $C$ must intersect all longest cycles.\qedhere\end{enumerate}\end{proof}
We will also need a technical statement captured in the lemma below, giving a lower bound on the sum of distances between disjoint pairs along a cycle. 
\begin{lemma}\label{lem:distantpairs}
Let $k \ge 1$, $C$ be a cycle and $A, B$ disjoint segments of $C$.  Let $\{a_i, b_i\}$, $1 \le i \le k$ be pairwise disjoint pairs of vertices in $C$ such that $a_i \in V(A)$, $b_i \in V(B)$ for $1\le i\le k$.  Then $$\sum_{i=1}^k \mathrm{dist}_C(a_i, b_i) \ge \frac{k^2}{2}.$$  
\end{lemma}
To see where the lower bound in the lemma above comes from, consider the following extremal configuration: $C$ is a cycle of length $2k$ with vertices $a_1,a_2,\ldots,a_k,b_k,b_{k-1},\ldots,b_1$ in circular order. One can then easily check that the above sum of distances evaluates to exactly $$\sum_{i=1}^{k}\min\{2i-1,2k-2i+1\}=\left\lceil\frac{k^2}{2}\right\rceil.$$
\begin{proof}[Proof of Lemma~\ref{lem:distantpairs}]
Let $Z$ be the set consisting of the endpoints of $A$. Then $Z$ separates $a_i$ and $b_i$ in $C$ for every $i$. Thus  $$\mathrm{dist}_C(a_i, b_i) \geq \mathrm{dist}_C(a_i,Z)+\mathrm{dist}_C(b_i,Z)$$ for every $1 \le i \le k$.

Assume without loss of generality that $a_1,\ldots,a_k$ appear in this order on $C$ and that $Z=\{a_1,a_k\}$. Then 
\begin{align*}\sum_{i=1}^{k}\mathrm{dist}_C(a_i,Z)&=\sum_{i=1}^{k}\min\left(\mathrm{dist}_C(a_i,a_1), \mathrm{dist}_C(a_i,a_k)\right) \geq \sum_{i=1}^{k}\min(i-1, k-i) = \left \lfloor \frac{k-1}{2}\right\rfloor\left\lfloor \frac{k}{2}\right\rfloor,\end{align*}
and, similarly,
$$\sum_{i=1}^{k}\mathrm{dist}_C(b_i,Z) \geq \left \lfloor \frac{k-1}{2}\right\rfloor\left\lfloor \frac{k}{2}\right\rfloor +k.$$
It follows that 
$$\sum_{i=1}^k \mathrm{dist}_C(a_i, b_i) \ge 2 \left \lfloor \frac{k-1}{2}\right\rfloor\left\lfloor \frac{k}{2}\right\rfloor + k = \left\lceil\frac{k^2}{2}\right\rceil.$$
\end{proof}

The next tool for the proof of Theorem~\ref{thm:main} comes, perhaps surprisingly, from the seminal work of Robertson and Seymour~\cite{RS9} on graph minor structure theory. To state their result, we first need some additional terminology. A \emph{society} is a pair $(G, \Omega)$ such that $G$ is a graph and $\Omega$ is a cyclic permutation on a subset of vertices of $G$, which we denote by $V(\Omega)$.  A \emph{segment} of a society is a set of vertices $A \subseteq V(\Omega)$ which is contiguous in $\Omega$, that is there do not exist distinct vertices $x_1, x_2 \in A$, $y_1, y_2 \notin A$ such that $x_1, y_1, x_2, y_2$ occur in $\Omega$ in that cyclic order.  An \emph{$\Omega$-path} is a path in $G$ with distinct endpoints both in $V(\Omega)$ and no internal vertex in $V(\Omega)$. A \emph{transaction of order $k$} in $(G, \Omega)$ is a set $\zP = \{P_1, \dots, P_k\}$ of $k$ pairwise disjoint $\Omega$-paths such that there exist disjoint segments $A$, $B$ of $\Omega$ such that every path $P_i$ has one endpoint in $A$ and one endpoint in $B$. 

We further need to define a \emph{linear decomposition} of a society $(G, \Omega)$. This is a labeling $v_1, v_2, \dots, v_t$ of $V(\Omega)$ such that $v_1, v_2, \dots, v_t$ occur in that cyclic order on $\Omega$ along with subsets $(X_1, \dots, X_t)$ of $V(G)$ such that
\begin{itemize}
\item  $v_i \in X_i$ for all $i$ and $\bigcup_{i=1}^t X_i = V(G)$.  For every edge $e = uv \in E(G)$, there exists an index $i$ such that $\{u, v\} \subseteq X_i$.  
\item For every vertex $x \in V(G)$ the set of indices $i$ for which $x\in X_i$ forms a subinterval of $[t]$.  
\end{itemize}
The \emph{adhesion} of the decomposition is defined as
\begin{equation*}
\max_{1\le i \le t-1}|X_i \cap X_{i+1}|.
\end{equation*}
Note that in \cite{RS9}, the adhesion of the decomposition is referred to as the ``depth''; we prefer the notation \emph{adhesion} in order to be consistent with the related notion for tree decompositions as we shall see in the next section.  With this terminology at hand, we are now ready to state the aforementioned result by Robertson and Seymour.
\begin{theorem}[\cite{RS9}, Theorem 8.1]\label{thm:bdedadhesion} Let $r \in \mathbb{N}$ and $(G, \Omega)$ a society.  If there does not exist a transaction $\zP$ with $|\zP| \ge r$, then there exists a linear decomposition of $(G, \Omega)$ with adhesion strictly less than $r$.  
\end{theorem}

Finally, we get to state our last ingredient for the proof of Theorem~\ref{thm:main}, namely a statement that guarantees the existence of a path traversing many edges of a matching spanned between two vertex-disjoint paths. This statement will be used in the proof of Theorem~\ref{thm:main} to deduce the existence of long paths in certain societies with large transactions, and we believe that it could be of independent interest.
\begin{lemma}\label{lem:matchings}
Let $P_1, P_2$ be two vertex-disjoint paths and let $M$ be a matching such that all edges of $M$ have exactly one endpoint in each of $P_1, P_2$. Then there exists a path $P$ in the graph $P_1\cup P_2\cup M$ such that $|E(P)\cap M|\ge 0.1\cdot |M|^{0.8}.$
\end{lemma}
 The proof of Lemma~\ref{lem:matchings} is somewhat more involved than the previous proofs in this section and is thus prepared and presented separately in the next section.
\section{Matchings between paths}\label{sec:matchings}
This section is devoted to proving Lemma~\ref{lem:matchings}. To prepare the proof, we first state two results from the literature that we will need. The following recent result of Liu, Yu and Zhang~\cite{Liu18}~is a weighted strengthening of a well-known result due to Jackson~\cite{jackson86} on long cycles in $3$-connected cubic graphs.
\begin{theorem}[cf.~Theorem~1.1. in~\cite{Liu18}]\label{thm:liu}
Let $G$ be a $3$-connected cubic graph equipped with a vertex-weighting $\omega:V(G)\rightarrow \mathbb{Z}_{\ge 0}$. Then for every pair of edges $e,f\in E(G)$ there exists a cycle $C$ in $G$ with $e,f\in E(C)$ such that 
$\omega(C)\ge 0.9\cdot  \omega(G)^{0.8}$.
\end{theorem}
The other literature result we are going to need is a weak form of the so-called \emph{$2$-separator theorem} originally due to Tutte~\cite{tutte01} and Cunningham and Edmonds~\cite{Cunningham80}, but in a different language using tree decompositions, as formulated for instance by Carmesin~\cite{Carmesin22}. The strongest form of this theorem gives \emph{existence} and \emph{uniqueness} of certain decompositions of $2$-connected graphs along $2$-separators into pieces which are either $3$-connected or cycles. Here we will only need the existence part, whose formal statement we now prepare.
Recall that a \emph{tree-decomposition} of a graph $G$ consists of a tree $T$ and a collection $(B_t)_{t\in V(T)}$ of subsets of $V(G)$, called \emph{bags}, such that
\begin{itemize}
    \item $\bigcup_{t\in V(T)}B_t=V(G)$,
    \item for every edge $e=uv\in E(G)$ there exists some $t\in V(T)$ such that $u,v \in B_t$, and
    \item for every $v\in V(G)$, the set $\{t\in  V(T)|v\in B_t\}$ induces a subtree of $T$.
\end{itemize}
An \emph{adhesion set} of a tree-decomposition $(T,(B_t)_{t\in V(G)})$ is a set of the form $B_{t_1}\cap B_{t_2}$ where $t_1t_2\in E(T)$. The \emph{adhesion} of the tree-decomposition is the maximum size of an adhesion set. Further, for any $t\in V(T)$, the \emph{torso} of $t$ is defined as the graph obtained from $G[B_t]$ by adding all edges of the form $xy$ where $x,y\in B_t\cap B_{t'}$ and $t'\in N_T(t)$ (i.e., all adhesion sets for edges incident with $t$ are made complete). Using this terminology, we can now state the aformentioned weak form of the $2$-separator theorem as follows.
\begin{theorem}[cf.~\cite{Carmesin22,Cunningham80}]\label{thm:2seps}
Let $G$ be a $2$-connected graph. Then $G$ admits a tree decomposition of adhesion at most $2$ such that every torso is a cycle or a $3$-connected graph.
\end{theorem}
To prove Lemma~\ref{lem:matchings} we first reduce it to the following statement on the existence of special cycles in certain almost-cubic graphs, from which Lemma~\ref{lem:matchings} can then be deduced. \begin{lemma}\label{lem:cubic}
Let $G$ be a $2$-connected graph with maximum degree at most $3$, let $e\in E(G)$ and suppose that every vertex in $G$ that is not an endpoint of $e$ has degree exactly $3$. Further let $C_0$ be a cycle in $G$ through $e$ such that $V(G)\setminus V(C_0)$ is an independent set. Then there exists a cycle $C$ in $G$ through $e$ such that $|V(C)\setminus V(C_0)|\ge 0.15\cdot |V(G)\setminus V(C_0)|^{0.8}$. 
\end{lemma}
Before jumping into the proof of Lemma~\ref{lem:cubic}, which is the main work in this section, we would like to motivate its statement by directly demonstrating how it implies Lemma~\ref{lem:matchings}.
\begin{proof}[Proof of Lemma~\ref{lem:matchings}, assuming Lemma~\ref{lem:cubic}]
Let $P_1, P_2$ be vertex-disjoint paths and $M$ a matching between $P_1$ and $P_2$. For the moment, let us assume that $|M|$ is divisible by $3$, we will later come back to the general case.
W.l.o.g. let us assume that the endpoints of $M$ cover all vertices of $P_1$ and $P_2$, since otherwise we could reduce to a smaller equivalent instance by contracting an edge of $P_1$ or $P_2$ that is adjacent with at most one edge in $M$. Let us enumerate the vertices of $P_1$ and $P_2$ in order along the paths as $u_1,\ldots,u_m$ and $v_1,\ldots,v_m$ respectively. Note that by our assumption we have $m=|M|$. We now define an auxiliary graph $G$ as follows: $G$ has vertex set $V(P_1)\cup \{w_1,\ldots,w_{m/3}\}$ where $w_1,\ldots,w_{m/3}$ are new vertices, and it has the following edges: All edges of $P_1$ plus an additional edge $e$ between $u_1$ and $u_m$, as well as for every $i\in [m/3]$ three edges from $w_i$ to the three unique neighbors of the vertices $v_{3i-2},v_{3i-1},v_{3i}$ in the matching $M$, all of which lie in $P_1$. It is easy to check that this resulting graph $G$ is cubic and $2$-connected, and that the cycle $C_0$ through $e$, defined by adding $e$ to the path $P_1$, has the property that $V(G)\setminus V(C_0)$ is an independent set in $G$. We can now apply Lemma~\ref{lem:cubic} to find that there exists a cycle $C$ in $G$ through $e=u_1u_\ell$ containing at least $0.15\cdot |V(G)\setminus V(C_0)|^{0.8}=0.15\cdot (m/3)^{0.8}$ of the vertices in $V(G)\setminus V(C_0)=\{w_1,\ldots,w_{m/3}\}$. In particular, the same is true for the path $P':=C-e$ with endpoints $u_1$ and $u_m$. We can now expand $P'$ into a path $P$ in $P_1\cup P_2\cup M$ by replacing each segment of $P'$ consisting of two incident edges of a vertex $w_i$ by the corresponding two matching edges in $M$ as well as the unique connection between their endpoints in $P_2$ (or, more specifically, in the subpath of $P_2$ with vertices $v_{3i-2}, v_{3i-1},v_{3i}$). Then $P$ is a path in $P_1\cup P_2\cup M$ which contains at least $2\cdot 0.15\cdot (m/3)^{0.8}\ge 0.124\cdot m^{0.8}$ edges of $M$. 

Summarizing, up to this point we have shown that in the case when $|M|$ is divisible by $3$, then we can guarantee a path in $P_1\cup P_2\cup M$ traversing at least $0.124\cdot |M|^{0.8}$ edges of $M$. In the general case, we can discard up to two edges of $M$ to reach a submatching of size divisible by three, and then guarantee a path in $P_1\cup P_2\cup M$ using at least $0.124\cdot (|M|-2)^{0.8}$ edges of $M$. One checks that this is already at least $0.1\cdot |M|^{0.8}$ unless $|M|\le 8$. However, in this case the statement holds trivially, since then $0.1\cdot |M|^{0.8}<1$. This concludes the proof of the lemma.
\end{proof}
To complete the proof of Lemma~\ref{lem:matchings} it remains to prove Lemma~\ref{lem:cubic}. As the last technical ingredient before its proof, we establish the following inequality for non-negative real numbers.
\begin{lemma}\label{lem:inequality}
Denote $r:=0.8$. For all $c\in (0,0.18]$, $k\in \mathbb{N}$ and non-negative reals $x,y_1,\ldots,y_k$ we have
$$\max\{0.9\cdot x^r+cy_1^r,\ldots,0.9\cdot x^r+cy_k^r,c(y_1^r+\cdots+y_k^r)\}\ge c(x+y_1+\cdots+y_k)^r.$$
\end{lemma}
\begin{proof}
In the proof, we will use the following fact:

For all real numbers $x,y,z\ge 0$ such that $y,z \ge 3x$ we have $y^r+z^r\ge (x+y+z)^r$. Indeed, this can be seen by noting that the difference $y^r+z^r- (x+y+z)^r$ is monotonically increasing in both $y$ and $z$, and by checking that the inequality holds when $y=z=3x$. In the following, we will refer to this inequality as $(\ast)$. Note that in particular, it implies $y^r+z^r\ge (y+z)^r$ for all $y,z\ge 0$.

So let $x,y_1,\cdots,y_k\ge 0$ be given arbitrarily, and w.l.o.g. assume that $y_1\ge y_2\ge \cdots \ge y_k$. Suppose first that $y_2+\cdots+y_k\le 6x$. We then have 
$0.9\cdot x^r+cy_1^r=c((0.9/c)^{1/r}x)^r+cy_1^r\ge c(7x)^r+cy_1^r\ge c(x+y_2+\cdots+y_k)^r+cy_1^r \ge c(x+y_1+y_2+\cdots+y_k)^r$, and so the desired inequality holds in this case.
Moving on, we may assume $y_2+\cdots+y_k>6x$. Let $i\in [k]$ be the smallest index such that $y_1+\cdots+y_i>3x$. Note that $y_2+\cdots+y_i\le y_1+\cdots+y_{i-1}\le 3x$ since $y_1\ge y_2\ge \cdots \ge y_k$ and by choice of $i$. 
Hence, we must have $y_{i+1}+\cdots+y_k=(y_2+\cdots+y_k)-(y_2+\cdots+y_i)\ge 6x-3x=3x$.
We can now use the inequality $(\ast)$ with $y=y_1+\cdots+y_i$ and $z=y_{i+1}+\cdots+y_k$, yielding
$$c(y_1^r+\cdots+y_k^r)\ge c(y^r+z^r)\ge c(x+y+z)^r=c(x+y_1+\cdots+y_k)^r.$$ This implies the assertion of the lemma, concluding the proof.
\end{proof}
Finally, we are ready for the proof of Lemma~\ref{lem:cubic}.
\begin{proof}[Proof of Lemma~\ref{lem:cubic}]
Towards a contradiction, suppose the statement is false and let $G$ (together with an edge $e\in E(G)$ and a cycle $C_0$ through $e$) be a smallest counterexample. Concretely, we assume $G$ is $2$-connected and has maximum degree at most $3$, $e\in E(G)$, every vertex of $G$ that is not an endpoint of $e$ has degree exactly $3$, and that $V(G)\setminus V(C_0)$ is an independent set. Further, we assume that there exists no cycle $C$ through $e$ in $G$ with $|V(C)\setminus V(C_0)|\ge 0.15\cdot |V(G)\setminus V(C_0)|^{0.8}$. Finally, we also assume that statement of the lemma holds for all graphs with less than $|V(G)|$ vertices. Our goal in the following is to derive a contradiction.

By Theorem~\ref{thm:2seps} there is a tree decomposition $(T,(B_t)_{t\in V(T)})$ of $G$ with adhesion at most $2$ and such that every torso is either a $3$-connected graph or a cycle. Note that the latter in particular implies $|B_t|\ge 3$ for every $t\in V(T)$. By definition of a tree decomposition, there must exist some tree vertex $t_0\in V(T)$ such that both endpoints of $e$ lie in the bag $B_{t_0}$. 
In the following, for every neighbor $t\in N_T(t_0)$ let us denote by $S(t)$ the maximal subtree of $T$ that contains $t$ but not $t_0$. 
We start with a simple but crucial observation regarding the interaction of $C_0$ and the tree decomposition.
\paragraph*{\textbf{Claim~1.}} For every $t\in N_T(t_0)$ we have $|B_t\cap B_{t_0}|=2$ and the intersection of $C_0$ with $G\left[\bigcup_{s\in S(t)}B_s\right]$ is a segment of the path $C_0-e$ of length at least two connecting the two vertices in $B_t\cap B_{t_0}$.
\begin{claimproof}
 The definition of a tree-decomposition implies that the adhesion set $B_t\cap B_{t_0}$ separates $\bigcup_{s\in S(t)}B_s$ and $\bigcup_{s\in V(T)\setminus S(t)}B_s$ in~$G$. As the latter sets both have size at least~$3$, since $G$ is $2$-connected, and because the tree decomposition has adhesion at most two, this implies $|B_t\cap B_{t_0}|= 2$. Next, pick (arbitrarily) some vertex $x\in \left(\bigcup_{s\in S(t)}B_s\right)\setminus (B_t\cap B_{t_0})=(\bigcup_{s\in S(t)}B_s)\setminus B_{t_0}$. Since the endpoints of $e$ are contained in $B_{t_0}$, $x$ must have degree exactly $3$ in $G$. In particular, as $|B_t\cap B_{t_0}|=2$ and all neighbors of $x$ lie in $\bigcup_{s\in S(t)}B_s$, the vertex $x$ must have some neighbor $y\in \left(\bigcup_{s\in S(t)}B_s\right)\setminus (B_t\cap B_{t_0})$. Recalling the assumption that $V(G)\setminus V(C_0)$ is an independent set in $G$, this implies that $\{x,y\}\cap V(C_0)\neq \emptyset$. As $C_0$ also contains the two endpoints of $e$, which lie in $\bigcup_{s\in V(T)\setminus S(t)}B_s$, it follows that $C_0$ must pass through both vertices of the separator $B_t\cap B_{t_0}$. Because it also must pass through $x$ or $y$, the intersection of $C_0$ with $G\left[\bigcup_{s\in S(t)}B_s\right]$ must be a path of length at least two that connects the two elements of $B_t\cap B_{t_0}$. Since the endpoints of $e$ both do not  lie in $(\bigcup_{s\in S(t)}B_s)\setminus (B_t\cap B_{t_0})$ while this is the case for at least one endpoint of all edges of the intersection of $C_0$ with $G\left[\bigcup_{s\in S(t)}B_s\right]$, this path is in particular contained in $C_0-e$, as claimed. This concludes the proof.
\end{claimproof}
For every neighbor $t$ of $t_0$ in $T$, let us denote by $P_t$ the segment of $C_0-e$ that forms the intersection of $C_0$ with $G\left[\bigcup_{s\in S(t)}B_s\right]$. Note that, by definition of a tree decomposition, for any two distinct $t,t'\in N_T(t_0)$ we have $V(P_t)\cap V(P_{t'})\subseteq B_t\cap B_{t'}=B_t\cap B_{t'}\cap B_{t_0}$, and hence the paths $(P_t)_{t\in N_T(t_0)}$ are pairwise internally vertex disjoint.

These facts now allow us to put an ordering $t_1,\ldots,t_k$ on the neighbors of $t_0$ in $T$ such that the paths $P_{t_1},P_{t_2},\ldots,P_{t_k}$ appear in this order along the path $C_0-e$. In the following, let us abbreviate $P_i:=P_{t_i}$ for every $i\in [k]$. We now prepare the construction of the cycle $C$ required by the statement of the lemma with three more claims. In the following, let us denote by $H$ the torso of $t_0$, i.e., the graph obtained from $G[B_{t_0}]$ by adding an edge $e_i$ between the two elements of $B_{t_i}\cap B_{t_0}$ for every $i\in [k]$. Recall that by our choice of the tree-decomposition, $H$ is a cycle or a $3$-connected graph. Further for every $i\in [k]$ let us denote by $G_i$ the graph obtained from $G\left[\bigcup_{s\in S(t_i)}B_s\right]$ by adding the edge, which we denoted above by $e_i$, between the vertices in $B_{t_i}\cap B_{t_0}$.
\paragraph*{\textbf{Claim~2.}} $e_i\neq e$ for every $i\in [k]$. 
\begin{claimproof}
    Suppose $e=e_i$ for some $i$. By definition of $e_i$, the removal of the endpoints of $e=e_i$ then disconnects $G$. On the other hand, removing the endpoints of $e$ from $C_0$ leaves a path $R$ (thus a connected graph), and each vertex in the independent set $V(G)\setminus V(C_0)$ has degree three in $G$ and hence must have at least one neighbor in $R$. Hence, removing the endpoints of $e$ from $G$ leaves a connected graph, a contradiction. This shows $e\neq e_i$, as desired.
\end{claimproof}
\paragraph*{\textbf{Claim~3.}} For every edge $f\in E(H)$ there is a cycle $C_{f}$ in $H$ passing through $e$ and $f$ such that $|V(C_{f})\setminus V(C_0)|\ge 0.9\cdot|V(H)\setminus V(C_0)|^{0.8}$.
\begin{claimproof}
If $H$ is a cycle, then this claim holds trivially. Next, suppose that $H$ is $3$-connected. We claim that $H$ also has maximum degree at most $3$, i.e., is a cubic graph. To see this, consider the graph $H^\ast$ obtained from $H$ by replacing for every $i\in [k]$ the edge in $H$ between the two elements of $B_{t_i}\cap B_{t_0}$ by the path $P_i$. Since $P_1,\ldots,P_k$ are pairwise internally disjoint and intersect $V(H)$ only in their endpoints, we have that $H^\ast$ is isomorphic to  a subdivision of $H$, in particular, we have $\Delta(H)= \Delta(H^\ast)$. However, it follows by definition of the paths $P_i$ and of $H$ that $H^\ast$ is a subgraph of $G$, and so $\Delta(H)=\Delta(H^\ast)\le \Delta(G)\le 3$, as desired. Hence, $H$ is a cubic $3$-connected graph. The claim now follows immediately from Theorem~\ref{thm:liu} by considering the weight function $\omega:V(H)\rightarrow \mathbb{Z}_{\ge 0}$, defined as $\omega(v):=1$ for $v\in V(H)\setminus V(C_0)$ and $\omega(v):=0$ for $v\in V(C_0)\cap V(H)$. 
\end{claimproof}
\paragraph*{\textbf{Claim~4.}} For every $i\in [
k]$ there exists a path $Q_i$ in $G_i-e_i$ with the same endpoints as $P_i$ such that $|V(Q_i)\setminus V(C_0)|\ge 0.15\cdot |V(G_i)\setminus V(C_0)|^{0.8}$.
\begin{claimproof}
In the following, let us denote by $C_i$ the cycle obtained from $P_i$ by adding the edge $e_i$. We claim that $G_i$ together with the edge $e_i$ and the cycle $C_i$ satisfies all the conditions of the lemma. Indeed, one easily observes that $G_i$ is $2$-connected, and that all vertices in $G_i$ distinct from the endpoints of $e_i$ have degree $3$ in $G$ and thus also in $G_i$. Further, note that also the two endpoints of $e_i$ have degree at most $3$ in $G_i$: Since $G$ is $2$-connected, and the endpoints of $e_i$ separate $V(G_i)$ and $V(G)\setminus V(G_i)$ in $G$, both endpoints of $e_i$ must have at least one neighbor in $V(G)\setminus V(G_i)$. Hence, $d_{G_i}(v)\le d_{G_i-e}(v)+1\le (d_G(v)-1)+1=d_G(v)\le 3$ for each endpoint $v$ of~$e_i$. Hence, $G_i$ has maximum degree at most $3$. Finally, we have that $V(G_i)\setminus V(C_i)\subseteq V(G)\setminus V(C_0)$. Hence, $V(G_i)\setminus V(C_i)$ is an independent set in $G$ and thus also in $G_i$, since both endpoints of $e_i$ are not contained in $V(G_i)\setminus V(C_i)$. This now shows that $G_i$ together with the edge $e_i$ and the cycle $C_i$ indeed meets all conditions of the lemma. Since further $|V(G_i)|<|V(G)|$ by definition of $G_i$ and we assumed in the beginning that $G$ is a smallest counterexample to the lemma, the assertion of the lemma thus holds for $G_i$. This means we find a cycle $C_i'$ in $G_i$
 through $e_i$ such that $|V(C_i')\setminus V(C_i)|\ge 0.15\cdot |V(G_i)\setminus V(C_i)|^{0.8}$. Defining $Q_i:=C_i'-e_i$ and since $V(C_0)\cap V(G_i)=V(C_i)\cap V(G_i)$ this reformulates to $|V(Q_i)\setminus V(C_0)|\ge 0.15\cdot |V(G_i)\setminus V(C_0)|^{0.8}$, proving the claim. \end{claimproof}
 We are now ready to conclude the proof using Lemma~\ref{lem:inequality}. For each $i\in [k]$, apply Claim~3 with $f:=e_i$ to obtain a cycle $C_{e_i}$ in $H$ through $e$ and $e_i$ such that $|V(C_{e_i})\setminus V(C_0)|\ge 0.9\cdot |V(H)\setminus V(C_0)|^{0.8}$. Let $C_i^\ast$ denote the cycle in $G$ obtained from the cycle $C_{e_i}$ in $H$ by replacing every occurrence of an edge $e_j$ with $j\neq i$ by the path $P_j$, and replacing $e_i$ with the path $Q_i$. These replacing paths are internally disjoint from each other and from $V(C_{e_i})$ (since $V(G_i)\cap V(G_j)\subseteq V(G_i)\cap V(H)=e_i$ for all $i,j\in [k]$). Thus, $C_i^\ast$ is indeed a cycle in $G$ for every $i$, satisfying
 $$|V(C_i^\ast)\setminus V(C_0)|\ge |V(C_{e_i})\setminus V(C_0)|+|V(Q_i)\setminus V(C_0)|\ge 0.9\cdot |V(H)\setminus V(C_0)|^{0.8}+0.15\cdot |V(G_i)\setminus V(C_0)|^{0.8}.$$
 Observe that since $e\in C_{e_i}$ and since none of the edges $e_1,\ldots,e_k$ which we replace by paths equals $e$ (Claim~$2$), we have that $C_i^\ast$ passes through $e$.
 
 Further, let $C^\ast$ denote the cycle obtained from $C_0$ by replacing, for every $i\in [k]$, the segment $P_i$ of $C_0-e$ by the path $Q_i$. Since the paths $Q_1,\ldots,Q_k$ are pairwise internally disjoint, are contained in $G$ and satisfy $V(Q_i)\cap V(C_0)\subseteq V(G_i)\cap V(C_0)=V(P_i)$ for all $i$, this is a cycle in $G$ through $e$. We further have
 $$|V(C^\ast)\setminus V(C_0)|\ge \sum_{i=1}^{k}|V(Q_i)\setminus V(C_0)|\ge 0.15 \sum_{i=1}^{k}|V(G_i)\setminus V(C_0)|^{0.8}.$$ Finally, let us define $x:=|V(H)\setminus V(C_0)|$ and $y_i:=|V(G_i)\setminus V(C_0)|$ for all $i\in [k]$. Set $c:=0.15$ and $r:=0.8$. If $k\ge 1$, then by the above inequalities and Lemma~\ref{lem:inequality} there exists a cycle $C$ in $G$ through $e$ such that
 $$|V(C)\setminus V(C_0)|\ge \max\{0.9\cdot x^r+cy_1^r,\ldots,0.9\cdot x^r+cy_k^r,c(y_1^r+\cdots+y_k^r)\}\ge c(x+y_1+\cdots+y_k)^{r}.$$
 Note that using the properties of a tree-decomposition and the definition of $H,G_1,\ldots,G_k$, we have $V(H)\cup V(G_1)\cup\cdots V(G_k)=V(G)$ and $V(G_i)\cap V(G_j)\subseteq V(G_i)\cap V(H)\subseteq V(C_0)$ for all $i,j\in [k]$. This implies that the sets $V(H)\setminus V(C_0),V(G_1)\setminus V(C_0),\ldots,V(G_k)\setminus V(C_0)$ are pairwise disjoint, and hence
 $$x+y_1+\cdots+y_k=\left|(V(H)\cup V(G_1)\cup\cdots V(G_k))\setminus V(C_0)\right|=|V(G)\setminus V(C_0)|.$$ Hence, the above inequality yields that $C$ satisfies $|V(C)\setminus V(C_0)|\ge 0.15\cdot|V(G)\setminus V(C_0)|^{0.8}$, contradicting our initial assumptions on $G$ and therefore ruling out the case $k\ge 1$. It remains to deal with the case $k=0$. In this case, $V(T)=\{t_0\}$ and hence $G=H$. Then by Claim~3 there exists a cycle $C$ through $e$ in $G$ such that $|V(C)\setminus V(C_0)|\ge 0.9\cdot |V(G)\setminus V(C_0)|^{0.8}\ge 0.15\cdot |V(G)\setminus V(C_0)|^{0.8}$. This establishes the assertion of the lemma for $(G,e,C_0)$ also in this second case, again contradicting our initial assumptions and completing the proof.\end{proof}
\section{Proofs of Theorem~\ref{thm:main} and Corollary~\ref{cor:lovasz}}\label{sec:thm}
Being equipped with all the necessary auxiliary statements, we can now finally present the proof of our main results, Theorem~\ref{thm:main} and Corollary~\ref{cor:lovasz}.
\newpage
\begin{proof}[Proof of Theorem~\ref{thm:main}]\noindent
\begin{enumerate}[(i)]
    \item Recall that $n$ denotes the number of vertices and $\ell$ the maximum length of a path in $G$. By Lemma~\ref{lemma:nicehitting}, we either have $\mathrm{lpt}(G)=1$ or there exists a cycle which intersects all longest paths. Since the claim is already verified in the first case, we may proceed with the second case. Let $C$ denote a cycle that intersects all longest paths in $G$ and is of minimum length subject to this property. We can assume that $|C|> \min\{\sqrt{8n},33\cdot\ell^{5/9}\}$, for otherwise $V(C)$ can be used as the desired longest path transversal. We now distinguish two cases depending on whether $C$ is geodetic.
    \paragraph*{Case~1.}
    $C$ is not geodetic. Then there exist distinct vertices $u,v\in V(C)$ with $\mathrm{dist}_G(u,v)<\mathrm{dist}_C(u,v)$. Let $P$ be a shortest $u$-$v$-path in $G$. While $P$ can intersect $C$ several times, by the triangle inequality there always has to exist a subpath $P'$ of $P$ which is internally disjoint from $C$ and connects two points $x,y$ of $C$ such that $|P'|<\mathrm{dist}_C(x,y)$. Therefore $C\cup P'$ is the union of two cycles $C_1,C_2$ with $C_1\cap C_2=P'$ such that both $C_1$ and $C_2$ have length less than $C$. 
By our choice of $C$, this means that $V(C_1)$ and $V(C_2)$ cannot be longest path transversals, i.e., there exist longest paths $P_1, P_2$ in $G$ such that $V(P_i)\cap V(C_i)=\emptyset$ for $i=1,2$. By Lemma~\ref{lem:separate} there exist subsets $K_1, K_2\subseteq V(G)$, each of size at most $\sqrt{2\ell}$, such that $K_i$ separates $V(C_i)$ and $V(P_i)$ in $G$ for $i=1,2$. Since every longest path in $G$ must intersect $C$ (and thus $C_1$ or $C_2$) as well as both $P_1$ and $P_2$ (since longest paths in $G$ pairwise intersect), it follows that every longest path must also intersect $K_1$ or $K_2$. Hence, $K:=K_1\cup K_2$ forms a longest path transversal in $G$ with size at most $2\sqrt{2\ell}=\sqrt{8\ell}\le \min\{\sqrt{8n},33\cdot\ell^{5/9}\}$. In particular, $\mathrm{lpt}(G)\le \min\{\sqrt{8n},33\cdot\ell^{5/9}\}$, concluding the proof in Case~1.
\paragraph*{Case~2.} $C$ is a geodetic cycle. Let $(G, \Omega)$ be the society with $V(\Omega) = V(C)$ and the cyclic ordering of $\Omega$ given by the cycle $C$. Let $\zP$ be a transaction for $(G,\Omega)$ of maximum size. The reason for considering such  $\mathcal{P}$ is that its size yields an upper bound for $\mathrm{lpt}(G)$.
\paragraph*{Claim~1.} We have $\mathrm{lpt}(G)\le 2|\mathcal{P}|+1$.
\begin{claimproof}
By Theorem~\ref{thm:bdedadhesion} there exists a linear decomposition $v_1, \dots, v_t$, $(X_1, \dots, X_t)$ of $(G, \Omega)$ of adhesion at most $|\mathcal{P}|$.  For every longest path $L$ in $G$, let $I_L := \{i: V(L) \cap X_i \neq \emptyset\}$.  By the connectivity of $L$ and the definition of a linear decomposition, we have that $I_L$ must be a subinterval of $[t]$.  As every pair of longest paths intersects, we have that $I_L \cap I_{L'} \neq \emptyset$ for distinct longest paths $L, L'$.  By the Helly property of intervals, we conclude that there exists an index $i$ such that $V(L)\cap X_i\neq \emptyset$ for every longest path $L$ in $G$.

We now claim that the set $S := (X_i \cap X_{i-1}) \cup (X_i \cap X_{i+1}) \cup \{v_i\}$ (set $X_{t+1}:=X_{0}:=\emptyset$ in the boundary cases $i\in \{1,t\}$) intersects every longest path in $G$.

Indeed, consider any longest path $L$ and note that the definition of a linear decomposition implies that $(X_i\cap X_{i-1})\cup (X_i\cap X_{i+1})$ separates $X_i$ from $\bigcup_{j\neq i}X_j$ in $G$. As $V(L)\cap X_i\neq \emptyset$ by the above and $L$ is connected, this implies that $V(L)\cap S\neq \emptyset$ or $V(L)\cap \bigcup_{j\neq i}X_j=\emptyset$, but in the latter case we immediately obtain (since $v_j\in X_j$, $j=1,\ldots,t$ and $V(C)=\{v_1,\ldots,v_t\}$) that $V(L)\cap V(C)\subseteq \{v_i\}$. But since $V(C)$ is a longest path transversal, we must then have $V(L)\cap V(C)=\{v_i\}$. Hence $V(L)\cap S \supseteq \{v_i\}\neq\emptyset$ also in this case. Hence, $S$ is indeed a longest path transversal.

Since the adhesion of the linear decomposition is at most $|\mathcal{P}|$, we obtain $\mathrm{lpt}(G)\le |S|\le 2|\mathcal{P}|+1$, as desired.
\end{claimproof}
With Claim~1 at hand, our goal in the remainder is to upper bound $|\mathcal{P}|$.
\paragraph*{Claim~2.} We have $|\mathcal{P}|\le 16\cdot \ell^{5/9}$.
\begin{claimproof}
By definition of a transaction and since $\Omega$ corresponds to the cyclic order on the cycle $C$, we find that there exist vertex-disjoint subpaths $P_1$ and $P_2$ of $C$ such that every $P\in\mathcal{P}$ has one endpoint in $V(P_1)$ and one endpoint in $V(P_2)$. 

Let $p = |\mathcal{P}|$.
For $i\in\{1,2\}$ let $R_i\subseteq P_i$ denote the subpath of $P_i$ consisting of all vertices of distance (on $P_i$) at least $\lfloor\frac{5p}{36}\rfloor$ from both endpoints of $P_i$. Let further $\mathcal{M}\subseteq \mathcal{P}$ consist of all those $P\in \mathcal{P}$ that have one endpoint in $R_1$ and one endpoint in $R_2$. Note that $|(V(P_1)\cup V(P_2))\setminus (V(R_1)\cup V(R_2))|\le 4\cdot \lfloor\frac{5p}{36}\rfloor\le\frac{5p}{9}$ by definition of $R_1, R_2$. Since the paths in $\mathcal{P}$ are pairwise disjoint, this implies that $|\mathcal{M}|\ge |\mathcal{P}|-|(V(P_1)\cup V(P_2))\setminus (V(R_1)\cup V(R_2))|\ge p-\frac{5p}{9}=\frac{4p}{9}$. Now consider any path $P\in \mathcal{M}$ and let $x,y$ denote its endpoints. Since $C$ is geodetic and by definition of $R_1, R_2$, we must have
$$|P|\ge \mathrm{dist}_G(x,y)=\mathrm{dist}_C(x,y)\ge 2\cdot \left\lfloor\frac{5p}{36}\right\rfloor>\frac{5p}{18}-2.$$
Next, define a matching $M$ between the paths $R_1$ and $R_2$ obtained by matching the endpoints of the members of $\mathcal{M}$ to each other. By Lemma~\ref{lem:matchings} we find that there exists a path in $R_1\cup R_2\cup M$ using at least $0.1\cdot |M|^{0.8}$ edges of $M$. By expanding the matching edges in $M$ into the corresponding paths in $\mathcal{M}$, we obtain a path in $G$ of length at least
$$0.1\cdot |M|^{0.8}\cdot \left(\frac{5p}{18}-2\right)\ge 0.1\cdot \left(\frac{4p}{9}\right)^{0.8}\cdot \left(\frac{5p}{18}-2\right)\ge \left(1-\frac{36}{5p}\right)\cdot 0.014\cdot p^{1.8}.$$
Hence, we have $\ell\ge \left(1-\frac{36}{5p}\right)\cdot 0.014\cdot p^{9/5}$. If we had $p>16\cdot \ell^{5/9}$, then the above would yield $\ell>\left(1-\frac{36}{80}\right)\cdot 0.014\cdot 16^{9/5} \cdot \ell>\ell,$ a contradiction. Hence, we must have $|\mathcal{P}|=p\le 16\cdot \ell^{5/9}$, as desired.
\end{claimproof}
Claim~1 and~2 immediately yield $\mathrm{lpt}(G)\le 2|\mathcal{P}|+1\le 33\cdot \ell^{5/9}$. Hence, it remains to consider the case when $\min\{\sqrt{8n},33\cdot\ell^{5/9}\}=\sqrt{8n}$, where we have to prove that $\mathrm{lpt}(G)\le \sqrt{8n}$. We now claim that $|\mathcal{P}|\le \lfloor\sqrt{2n} \rfloor-1$. Towards a contradiction suppose that $|\mathcal{P}|\ge \lfloor\sqrt{2n} \rfloor$ and let $\mathcal{Q}\subseteq \mathcal{P}$ be such that $|\mathcal{Q}|=\lfloor \sqrt{2n}\rfloor$.

Given that $C$ is geodetic, each element $Q \in \mathcal{Q}$ is at least as long as the distance between its endpoints in $C$.  But then by Lemma~\ref{lem:distantpairs}, the sum of the lengths of the paths $Q \in \mathcal{Q}$ is at least $\lfloor \sqrt{2n} \rfloor^2/2$.  Thus $\left| \bigcup_{Q \in \mathcal{Q}} V(Q) \setminus V(C) \right|\ge \left(\sum_{Q\in \mathcal{Q}}|Q|\right)-|\mathcal{Q}| \ge \lfloor \sqrt{2n}\rfloor ^2/2 - \lfloor \sqrt{2n} \rfloor$.  Accounting for the vertices of $C$ and using $|C|>\min\{\sqrt{8n},33\cdot\ell^{5/9}\}=\sqrt{8n}$, we see that 
\begin{align*}
|V(G)| &> \frac{(\sqrt{2n} - 1)^2}{2} - \sqrt{2n} + 2\sqrt{2n}
 =n +\frac{1}{2},
 \end{align*}
 a contradiction. Hence, we indeed have $|\mathcal{P}|\le \lfloor\sqrt{2n}\rfloor-1$. Claim~1 now yields $\mathrm{lpt}(G)\le 2|\mathcal{P}|+1\le 2\lfloor\sqrt{2n}\rfloor-1\le \sqrt{8n}=\min\{\sqrt{8n},33\cdot \ell^{5/9}\}$, as desired. Having established the assertion of part~$(i)$ of the lemma in all cases, we conclude the proof. 
    \item Let $G$ be a $2$-connected graph on $n$ vertices. By Lemma~\ref{lem:intersect}, any two longest cycles in $G$ intersect each other, and by Lemma~\ref{lemma:nicehitting} there exists a cycle hitting all longest cycles. The proof then continues virtually identically to those parts of the proof for $(i)$ that concern bounding the transversal size in terms of $n$, subject to changing ``longest path'' to ``longest cycle'' at all relevant places
. The proof is thus omitted. \qedhere
\end{enumerate}
\end{proof}
Finally, we can also swiftly deduce Corollary~\ref{cor:lovasz} from Theorem~\ref{thm:main}, $(i)$.
\begin{proof}[Proof of Corollary~\ref{cor:lovasz}]
Let $G$ be a connected vertex-transitive graph on $n\ge 3$ vertices. We shall first show that $G$ contains a path of length at least $\Omega(n^{9/14})$, and then deduce from this that $G$ contains also a cycle of length at least $\Omega(n^{9/14})$, which then concludes the proof.  Let $\ell$ denote the maximum length of a path in $G$. According to~\cite[Proposition~5.1]{groen24} we then have $\ell\ge \frac{n}{\mathrm{lpt}(G)}-1$. By Theorem~\ref{thm:main}, we have $\mathrm{lpt}(G)\le 33\cdot \ell^{5/9}$. Rearranging, we find
$$n\le (\ell+1)\mathrm{lpt}(G)\le 33(\ell+1)\ell^{5/9}\le 66\cdot \ell^{14/9}.$$
This implies that $\ell\ge \Omega(n^{9/14})$.
Next, we shall show that $G$ also contains a cycle of length at least $\Omega(n^{9/14})$. To do so, denote by $d\in \mathbb{N}$ the degree of some (and thus every) vertex in $G$. Since $n\ge 3$, we have $d\ge 2$. If $d=2$, then $G$ is a cycle of length $n$, and thus trivially contains a cycle of length $\Omega(n^{9/14})$. Moving on, we may thus assume $d\ge 3$. By a result of Watkins~\cite[Theorem~3]{watkins}, $G$ then has vertex-connectivity strictly greater than $\frac{2}{3}d\ge 2$, so it follows that $G$ is a $3$-connected graph. Let $\ell'$ denote the maximum length of a cycle in $G$. By a result of Bondy and Locke~\cite{bondy}, applicable since $G$ is $3$-connected, we then have $\ell'\ge \frac{2}{5}\ell\ge \Omega(n^{9/14})$. This concludes the proof of the corollary.
\end{proof}

\bibliographystyle{plain}
\bibliography{biblio.bib}
\end{document}